\documentclass[12pt]{amsart}
\title{Strongly stable fields}
\author{Will Johnson}
\email{willjohnson@fudan.edu.cn}
\address{Department of Mathematics \\ Fudan University \\ 220 Handan Road, Shanghai, China}
\author{Jinhe Ye}
\email{jinheye@hku.hk}
\address{Department of Mathematics, Run Run Shaw Building, The University of Hong Kong, Pokfulam, Hong Kong}
\thanks{The first author was supported by the National Natural Science Foundation
  of China (Grant No.\@ W2532009).  The second author was supported by his startup fund from the University of Hong Kong.}
\date{\today}

\usepackage{amsmath, amssymb, amsthm}    	
\usepackage{fullpage} 	
\usepackage{amscd}
\usepackage{hyperref}
\usepackage[all]{xy}
\usepackage{centernot,enumitem}

\usepackage{CJKutf8}

\def\Ind#1#2{#1\setbox0=\hbox{$#1x$}\kern\wd0\hbox to 0pt{\hss$#1\mid$\hss}
\lower.9\ht0\hbox to 0pt{\hss$#1\smile$\hss}\kern\wd0}

\def\ind{\mathop{\mathpalette\Ind{}}}

\def\notind#1#2{#1\setbox0=\hbox{$#1x$}\kern\wd0
\hbox to 0pt{\mathchardef\nn=12854\hss$#1\nn$\kern1.4\wd0\hss}
\hbox to 0pt{\hss$#1\mid$\hss}\lower.9\ht0 \hbox to 0pt{\hss$#1\smile$\hss}\kern\wd0}

\DeclareMathOperator{\wt}{wt}

\newcommand{\Gal}{\operatorname{Gal}}

\newcommand{\characteristic}{\operatorname{char}}

\newcommand{\Aut}{\operatorname{Aut}}

\newcommand{\acl}{\operatorname{acl}}
\newcommand{\dcl}{\operatorname{dcl}}
\newcommand{\tp}{\operatorname{tp}}

\newcommand{\Stab}{\operatorname{Stab}}

\newtheorem{theorem}{Theorem}[section] 
\newtheorem{lemma}[theorem]{Lemma}

\newtheorem{fact}[theorem]{Fact}

\newtheorem*{theorem-star}{Theorem}
\newtheorem*{lemma-star}{Theorem}
\newtheorem*{conjecture-star}{Conjecture}

\theoremstyle{definition}
\newtheorem{definition}[theorem]{Definition}

\newtheorem{remark}[theorem]{Remark}

\newtheorem*{ai}{AI statement}

\newcommand{\Qq}{\mathbb{Q}}

\newcommand{\Zz}{\mathbb{Z}}

\newcommand{\Mm}{\mathbb{M}}

\let\phi\varphi               

\begin{document}

\begin{abstract}
    Let $K$ be a field which is strongly stable, that is, stable and strongly dependent.  Then $K$ is finite or algebraically closed. This answers a question of Shelah.
\end{abstract}

\maketitle

\section{Introduction}

Recall that a theory $T$ is \emph{strongly dependent} (or \emph{strongly NIP}) if the following property does \emph{not} hold:
\begin{itemize} \item There is a model $M \models T$ containing countably many mutually indiscernible sequences $I_1$, $I_2$, $I_3$, \ldots and a finite tuple $a$ such that each sequence $I_j$ fails to be indiscernible over $a$.
\end{itemize}
A structure $M$ is strongly dependent if its complete theory is.  For other equivalent definitions, see \cite[Section 2]{dp-add} or \cite[Chapter 4]{NIPguide}.  Strongly dependent theories were introduced by Shelah~\cite{Shelah-Dependent} as an analogue of superstability for NIP theories:
\begin{quote}
    ``Superstable'' is to ``stable'' as ``strongly dependent'' is to ``NIP''.
\end{quote}
In particular, superstable theories are strongly dependent \cite[Corollary~4.35]{NIPguide} and strongly dependent theories are NIP (see \cite[Observation~4.13, Definition~4.23]{NIPguide}, or \cite[Observation~2.2, Definition~2.5]{firstDP}).

A theory is \emph{strongly stable} if it is both stable and strongly dependent.  The class of strongly stable theories is somewhat larger than the class of superstable theories.  For example, if $V_1 \supsetneq V_2 \supsetneq \cdots$ is a descending chain of $\Qq$-vector spaces, then the structure $(V_1,+,V_2,V_3,\ldots)$ is strongly stable but not superstable.  It turns out that a stable theory is strongly stable if and only if every complete type $p \in S_n(M)$ has finite weight $\wt(p)$ \cite[Corollary~2.13]{firstDP}.

There is a conjectural classification of strongly dependent fields, due to Halevi, Hasson, and Jahnke \cite{hhj1}.  So far, this conjectural classification is known only in the finite-dimensional (i.e., dp-finite) case \cite{dp-minimal,dp-finite-fields-vi}.

In this short note, we verify the conjecture in the case of strongly stable fields:
\begin{theorem}[{= Theorem~\ref{thm:main-copy}}] \label{thm:main}
    Let $K$ be a field which is strongly stable.  Then $K$ is finite or algebraically closed.
\end{theorem}
Theorem~\ref{thm:main} answers~\cite[Question 5.35(2)]{Shelah-Dependent} by Shelah. Our result generalizes a long chain of previous results, including the following:
\begin{itemize}
    \item Macintyre's theorem: $\omega$-stable fields are algebraically closed or finite \cite{Macintyre}.
    \item Cherlin and Shelah's theorem: superstable fields are algebraically closed or finite \cite{Cherlin-Shelah}.
    \item Krupi\'nski and Pillay's theorem: strongly stable fields are perfect, and if the generic type has weight 1, then the field is algebraically closed \cite{Krunpinski-Pillay}.
    \item Halevi and Palac\'in's theorem: dp-finite stable fields are algebraically closed or finite \cite{Halevi-Palacin}.  This is the case where there is a uniform finite bound on the weights of unary types.
    \item Shelah's theorem: strongly$^2$ stable fields are algebraically closed or finite \cite[Conclusion 5.39]{Shelah-Dependent}.
\end{itemize}
We follow the same general proof outline of many of these theorems: given a strongly stable infinite field $K$ which isn't algebraically closed, we pass to a finite extension which admits an Artin-Schreier extension or a Kummer extension, and then use some trickery to obtain a contradiction. The key ingredient in our argument is Lemma~\ref{sdep}, which isolates the essential mathematical content of the proof of~\cite[Proposition 2.3]{Krunpinski-Pillay}. It describes a configuration witnessing the failure of strong dependence, and the existence of a Kummer extension allows us to realize such a configuration.
\section{The proof}

Work in some highly saturated monster model $\Mm$.  We first recall three facts from the literature.
\begin{fact}[{\cite[Theorem 4.3]{Kaplan-Scanlon-Wagner}}] \label{fact:ksw}
    Suppose $\Mm$ is stable, or more generally NIP.  If $(K,+,\cdot)$ is an infinite definable field of characteristic $p>0$, then the Artin-Schreier map
    \begin{align*}
        f : K &\to K \\
        f(x) &= x^p - x
    \end{align*}
    is surjective, so $K$ has no Artin-Schreier extensions.
\end{fact}

\begin{fact}[{\cite[Theorem 5.10]{poizat}}] \label{poizat}
    Suppose $\Mm$ is stable and $(K,+,\cdot)$ is an infinite definable field.  Then the multiplicative group $K^\times$ and additive group $(K,+)$ are both connected, and have the same generic type.
\end{fact}
\begin{fact}[{\cite[Proposition~2.3]{Krunpinski-Pillay}}] \label{perfect}
If $\Mm$ is strongly stable, any definable field $(K,+,\cdot)$ is perfect.
\end{fact}
Krupi\'ski and Pillay's proof of \cite[Proposition~2.3]{Krunpinski-Pillay} can be generalized to give the following:
\begin{lemma} \label{sdep}
  Let $X, Y$ be definable sets, with $X$ infinite.  Let $f : X \times
  Y \to Y$ be a definable finite-to-one map.  Then the theory isn't
  strongly dependent.
\end{lemma}
\begin{proof}
  Naming parameters, we may assume $X,Y,f$ are 0-definable.
  
  Let $\{a_{i,j}\}_{i,j \in \omega}$ be an $\omega \times \omega$
  mutually indiscernible array of pairwise distinct elements of $X$.  For any $n$, we can build a sequence $b_0, b_1, b_2, \ldots, b_n \in Y$ such that
  \begin{equation*}
    b_i = f(a_{0,i},b_{i+1}) \text{ for all } i < n. \tag{$\ast$}
  \end{equation*}
  Indeed, take $b_n \in Y$ arbitrary and use ($\ast$) to recursively
  choose $b_{n-1}, b_{n-2}, \ldots, b_1, b_0$.
  
  By saturation, there is an infinite sequence $b_0, b_1, b_2, \ldots \in Y$
  such that
  \begin{equation*}
    b_i = f(a_{0,i},b_{i+1}) \text{ for all } i < \omega.
  \end{equation*}
  Since $f$ is finite-to-one, we have $a_{0,i},b_{i+1} \in \acl(b_i)$
  for each $i$.  Then by induction we get $\acl(b_0) \supseteq
  \{b_0,a_{0,0},b_1,a_{0,1},b_2,a_{0,2},\ldots\}$.  In particular,
  every element of column zero of $\{a_{i,j}\}_{i,j \in \omega}$ is
  algebraic over $b_0$.  Then no row of $\{a_{i,j}\}_{i,j \in \omega}$
  is $b_0$-indiscernible, so the theory isn't strongly dependent.
\end{proof}
\begin{lemma} \label{zeroth}
  Assume the theory is stable.  Let $(G,+)$ be a connected 0-definable
  group, with generic type $p$.  Let $a,c_1,\ldots,c_n$ be elements of
  $G$, and $b$ be some other element, such that
  \begin{itemize}
  \item $a$ is generic over $c_1,\ldots,c_n,b$, i.e., $a$ realizes
    $p|c_1c_2 \cdots c_n b$.
  \item $b \in \acl(a+c_1,a+c_2,\ldots,a+c_n)$.
  \end{itemize}
  Then $b \in \acl(c_2-c_1,c_3-c_1,c_4-c_1,\ldots,c_n-c_1)$.
\end{lemma}
The intuition here is that $(c_2-c_1,c_3-c_1,\ldots,c_n-c_1)$ is the canonical base of $\tp(a+c_1,a+c_2,\ldots,a+c_n/c_1,\ldots,c_n,b)$.
\begin{proof}
  Let $a' = a+c_1$.  Then $a'$ is generic over $c_1,\ldots,c_n,b$, so
  \begin{equation*}
    a' \ind_\varnothing \bar{c}b,
  \end{equation*}
  where $\bar{c}=(c_1,\ldots,c_n)$.  Let $u=(c_2-c_1,c_3-c_1,\ldots,c_n-c_1) \in \dcl(\bar{c})$.  By base monotonicity,
  \begin{equation*}
    a' \ind_u \bar{c}b.  \tag{$\dag$}
  \end{equation*}
  On the other hand
  \begin{align*}
    b &\in \acl(a+c_1,a+c_2,\ldots,a+c_n) \\
    &= \acl(a', a'+(c_2-c_1), a'+(c_3-c_1), \ldots, a'+(c_n-c_1)) \\
    &= \acl(a',c_2-c_1,\ldots,c_n-c_1) = \acl(a',u).
  \end{align*}
  By ($\dag$), $b \ind_u b$, so $b \in \acl(u)$.
\end{proof}
\begin{definition}
  Let $\ell$ be a prime.  A field $K$ is \emph{$\ell$-special} if $K$
  has a degree-$\ell$ Galois extension, and no prime $p < \ell$
  divides the degree of any Galois extension of $K$.
\end{definition}
\begin{lemma} \label{lem:to-special}
  Let $K$ be a field that isn't separably closed.  Then there is a
  prime $\ell$ and a finite extension $L/K$ which is $\ell$-special.
\end{lemma}
\begin{proof}
  This is well-known: take $\ell$ the smallest prime that divides the
  degree of any Galois extension of $K$.  Take $F/K$ a Galois extension
  with $\ell \mid [F : K]$.  Take $C \le \Gal(F/K)$ with $C \cong
  \Zz/\ell$.  Take $L$ the fixed field of $C$.
\end{proof}
We now turn to the proof of our main theorem
\begin{theorem} \label{thm:main-copy}
    Let $K$ be an infinite definable field in a strongly stable theory.  Then $K$ is algebraically closed.
\end{theorem}
The remainder of this section is devoted to the proof.  By Fact~\ref{perfect}, it suffices to show that $K$ is separably closed.  Suppose for the sake of contradiction that $K$ is not separably closed.  By Lemma~\ref{lem:to-special}, there is a prime $\ell$ and a finite extension $L/K$ which is $\ell$-special.  The field $L$ is still definable, and not separably closed.  Replacing $K$ with $L$, we may assume that $K$ is $\ell$-special.
\begin{remark}
We now have an infinite, strongly stable field $K$ that is $\ell$-special for some prime $\ell$.
We make several observations about this setting:
  \begin{enumerate}
  \item $K$ is perfect (Fact~\ref{perfect}).
  \item If $\ell = \characteristic(K)$, then the degree-$\ell$
    extension of $K$ is an Artin-Schreier extension, contradicting Fact~\ref{fact:ksw}.  Thus $\characteristic(K) \ne \ell$.
  \item If $P(x) \in K[x]$ has $\deg P < \ell$, then $P$ splits.
    Indeed, if $Q$ is any irreducible factor of $P$, then $Q$ is
    separable by (1), so $K$ has a finite separable extension of
    degree $\deg Q < \ell$, implying $\deg Q = 1$ because $K$ is
    $\ell$-special.
  \item Let $c \in K$ be arbitrary.  The polynomial $P(x) = (x+1)^\ell
    - x^\ell - c$ has degree $\ell-1>0$ (since $\ell \ne
    \characteristic(K)$), so it has a root: there is some $x$ with $(x+1)^\ell-x^\ell = c$.  Therefore \[K =
    \{(x+1)^\ell - x^\ell : x \in K\}.\]
  \item Taking $P(x) = x^{\ell-1} + x^{\ell-2} + \cdots + x + 1$ in (3) above, we
    see that $K$ contains the $\ell$th roots of unity.
  \item Then Kummer theory shows that the degree-$\ell$ Galois extension of $K$ has the form $K(\sqrt[\ell]{a})$ for some $a \in K$.
    In particular, the $\ell$th power map on $K$ is non-surjective:
    $(K^\times)^\ell \subsetneq K^\times$.  This is all standard.
  \end{enumerate}
\end{remark}
\begin{lemma} \label{first}
  Let $G$ be an additive subgroup of $K$, such that $x^\ell G
  \subseteq G$ for any $x \in K$.  Then $G = \{0\}$ or $G = K$.
\end{lemma}
\begin{proof}
  Suppose $u \in G$ is non-zero.  For any $x \in K$, the group $G$
  contains $(x+1)^\ell u$ and $x^\ell u$, as well as their difference
  $((x+1)^\ell-x^\ell)u$.  By point (4) above, $((x+1)^\ell-x^\ell)u$
  takes every value in $K$ as $x$ ranges over $K$, and so $G=K$.
\end{proof}
Since $K$ is an infinite stable field, the additive and multiplicative
groups are connected, and have the same generic type $p$ (Fact~\ref{poizat}).  Let $q$ be
the pushforward of $p$ under the $\ell$th power map.  (Here, $p$ and
$q$ are 0-definable global types.)  Since $K^\times$ is connected, the proper
subgroup $(K^\times)^\ell$ is non-generic, so $q$ is non-generic.  Let $\Stab(q)$ denote the additive stabilizer of $q$, that is, the stabilizer of $q$ in the additive group $(K,+)$.
\begin{lemma}
$\Stab(q)$ is trivial.
\end{lemma}
\begin{proof}
  Let $G$ be the additive stabilizer.  Then $G$ is a subgroup of
  $(K,+)$, type-definable over $\varnothing$.  If $c \in K^\times$ and
  $a \in \Stab(q)$, then $c^\ell \cdot q = q$, so
  \begin{equation*}
    c^\ell \cdot a + q = c^\ell \cdot a + c^\ell \cdot q = c^\ell \cdot (a + q) = c^\ell \cdot q = q,
  \end{equation*}
  meaning that $c^\ell \cdot a \in \Stab(q)$.  So we have shown
  \begin{equation*}
    c \in K^\times \implies c^\ell \cdot \Stab(q) \subseteq \Stab(q).
  \end{equation*}
  By Lemma~\ref{first}, $\Stab(q)$ is $\{0\}$ or $K$.  But in the
  latter case, $q$ is generic, a contradiction.
\end{proof}
It follows that for any $a \in K$, the canonical base $\operatorname{CB}(a+q)$ of the type $a+q$ is $a$.  Indeed, if $\sigma \in \Aut(\Mm)$, then $\sigma(a+q) = \sigma(a)+q$, which equals $a+q$ if and only if $\sigma(a)=a$.
\begin{lemma} \label{third}
  Let $a, b_0, b_1, b_2, \ldots$ be a Morley sequence of $p$ over $\varnothing$.  (Equivalently, $a, b_0, b_1, b_2, \ldots$ are an independent sequence of generics over $\varnothing$.)  Then there is some $N \ge 1$ such that
  \begin{equation*}
    a \in \dcl(a+b_0^\ell, a+b_1^\ell, a+b_2^\ell, \ldots, a+b_N^\ell).
  \end{equation*}
\end{lemma}
\begin{proof}
  The sequence $a+b_i^\ell$ is a Morley sequence of $a+q$, so its
  definable closure contains $\operatorname{CB}(a+q)=a$.
\end{proof}
Fix $N$ as in Lemma~\ref{third}.
  If $a, b_0, b_1, \ldots, b_N$ are independent generics over $\varnothing$ (i.e., a realization of $p^{\otimes(N+2)}|\varnothing$), then
  \begin{equation*}
    a \in \dcl(a+b_0^\ell, a+b_1^\ell, \ldots, a+b_N^\ell).
  \end{equation*}
  Consequently,
  \begin{equation*}
    \{b_0,b_1,\ldots,b_N\} \subseteq
    \acl(a+b_0^\ell,a+b_1^\ell,\ldots,a+b_N^\ell).
  \end{equation*}
  By Lemma~\ref{zeroth}, we have
  \begin{equation*}
    \{b_0,b_1,\ldots,b_N\} \subseteq \acl(b_1^\ell-b_0^\ell,b_2^\ell-b_0^\ell,\ldots,b_N^\ell-b_0^\ell).
  \end{equation*}

Let $\bar{u}$ be the $N$-tuple $(b_1^\ell-b_0^\ell,\ldots,b_N^\ell-b_0^\ell)$.
Let $\phi(x_0,\ldots,x_N,y_1,\ldots,y_N)$ be a formula (without
parameters) such that
\begin{itemize}
\item $\phi(\bar{b};\bar{u})$ is true.
\item $\phi(\Mm;\bar{c})$ is finite for any $N$-tuple $\bar{c}$.
\end{itemize}
Let $f : K^{N+1} \to K^N$ be the function $f(x_0,x_1,\ldots,x_N) = (x_1^\ell-x_0^\ell,x_2^\ell-x_0^\ell,\ldots,x_N^\ell-x_0^\ell)$, so that $f(\bar{b}) = \bar{u}$.  Let $D$ be the 0-definable set
\begin{equation*}
  D = \{\beta \in K^{N+1} : \Mm \models \phi(\beta,f(\beta))\}.
\end{equation*}
Then $D$ contains $(b_0,b_1,\ldots,b_N)$ which realizes
$p^{\otimes(N+1)}|\varnothing$, so $D$ is generic in $(K^{N+1},+)$.
Finitely many additive translates cover $K^{N+1}$.

Also, the function $f : D \to K^N$ is finite-to-one.  Therefore we can
cover $K^{N+1}$ with finitely many sets $D_1 \cup D_2 \cup \cdots \cup
D_m$ (namely, translates of $D$) each of which admits a finite-to-one definable map $D_i \to K^N$.
Shrinking the $D_i$ to be a partition, and gluing the maps, we get a
finite-to-one definable map $K^{N+1} \to K^N$, contradicting
Lemma~\ref{sdep}.  This contradiction completes the proof of Theorem~\ref{thm:main-copy}.

\begin{ai}
    Generative AI was not used to write this paper or generate any of the ideas presented here.
\end{ai}

\bibliographystyle{plain}
\bibliography{ref}

\begin{thebibliography}{10}

\bibitem{Cherlin-Shelah}
G.~Cherlin and S.~Shelah.
\newblock Superstable fields and groups.
\newblock {\em Ann. Math. Logic}, 18(3):227--270, 1980.

\bibitem{hhj1}
Yatir Halevi, Assaf Hasson, and Franziska Jahnke.
\newblock A conjectural classification of strongly dependent fields.
\newblock {\em Bulletin of Symbolic Logic}, 25(2):182--195, June 2019.

\bibitem{Halevi-Palacin}
Yatir Halevi and Daniel Palac\'in.
\newblock The dp-rank of abelian groups.
\newblock {\em J. Symb. Log.}, 84(3):957--986, 2019.

\bibitem{dp-finite-fields-vi}
Will Johnson.
\newblock Dp-finite fields {VI}: the dp-finite {S}helah conjecture.
\newblock {arXiv:2005.13989v1 [math.LO]}, May 2020.

\bibitem{dp-minimal}
Will Johnson.
\newblock The classification of dp-minimal and dp-small fields.
\newblock {\em J. Eur. Math. Soc. (JEMS)}, 25(2):467--513, 2023.

\bibitem{dp-add}
Itay Kaplan, Alf Onshuus, and Alexander Usvyatsov.
\newblock Additivity of the dp-rank.
\newblock {\em Trans. Amer. Math. Soc.}, 365(11):5783--5804, 2013.

\bibitem{Kaplan-Scanlon-Wagner}
Itay Kaplan, Thomas Scanlon, and Frank~O. Wagner.
\newblock Artin-{S}chreier extensions in {NIP} and simple fields.
\newblock {\em Israel J. Math.}, 185:141--153, 2011.

\bibitem{Krunpinski-Pillay}
Krzysztof Krupi\'nski and Anand Pillay.
\newblock On stable fields and weight.
\newblock {\em J. Inst. Math. Jussieu}, 10(2):349--358, 2011.

\bibitem{Macintyre}
Angus Macintyre.
\newblock On {$\omega \sb{1}$}-categorical theories of fields.
\newblock {\em Fund. Math.}, 71(1):1--25. (errata insert), 1971.

\bibitem{firstDP}
Alf Onshuus and Alexander Usvyatsov.
\newblock On dp-minimality, strong dependence and weight.
\newblock {\em J. Symbolic Logic}, 76(3):737--758, 2011.

\bibitem{poizat}
Bruno Poizat.
\newblock {\em Stable groups}, volume~87 of {\em Mathematical Surveys and
  Monographs}.
\newblock American Mathematical Society, Providence, RI, 2001.
\newblock Translated from the 1987 French original by Moses Gabriel Klein.

\bibitem{Shelah-Dependent}
Saharon Shelah.
\newblock Strongly dependent theories.
\newblock {\em Israel J. Math.}, 204(1):1--83, 2014.

\bibitem{NIPguide}
Pierre Simon.
\newblock {\em A guide to {NIP} theories}, volume~44 of {\em Lecture Notes in
  Logic}.
\newblock Association for Symbolic Logic, Chicago, IL; Cambridge Scientific
  Publishers, Cambridge, 2015.

\end{thebibliography}
\end{document}